\documentclass[10pt]{amsart}

\usepackage{amsmath, amsthm, amssymb, amsfonts}

\usepackage{geometry, url}

\newtheorem{proposition}{Proposition}
\newtheorem{lemma}{Lemma}
\newtheorem{corollary}{Corollary}
\newtheorem{theorem}{Theorem}

\theoremstyle{definition}
\newtheorem{notation}{Notation}

\theoremstyle{remark}
\newtheorem{remark}{Remark}

\title{Some results on continuous deformed free group factors}
\author{Adam Merberg}
\address{Department of Mathematics,
  University of California,
  Berkeley, CA, USA 94720}
\email{amerberg@math.berkeley.edu}
\begin{document}
\begin{abstract}
We construct a Fock space associated to a symmetric function $Q:U\times U \to (-1,1)$,
where $U$ is a nonempty open subset of $\mathbb R^j$ for some $j$. Namely, we will have
operator-valued distributions $a(x)$ and $a^+(y)$ satisfying
\begin{align*}
 a(x)a^+(y)-Q(x,y)a^+(y)a(x)=\delta(x-y).
\end{align*}
Analogous to the $q_{ij}$-Fock space of Bozejko and Speicher \cite{BS94}, we have field operators
arising as the sum of the creation and annihilation operators. These operators
generate a von Neumann algebra analogous to the free group factors, and
are factors which do not have property $\Gamma$. It was pointed out to us by an anonymous referee
that this is a special case of a theorem of Krolak \cite{Krolak}. 
\end{abstract}
\maketitle
\section{Introduction}\label{sec:intro}
In the study of operator algebras, much attention has been paid to the
canonical commutation relations (CCR) and the canonical anti-commutations (CAR).
Bo{\.z}ejko and Speicher \cite{BS91} considered an interpolation between these
relations. Specifically, for $q\in [-1,1]$, they constructed creation operators
$c^+(f)$ and annihilation operators $c(f)$ on a $q$-twisted Fock space $\mathcal
F_q(\mathcal H)$ satisfying the relations
\begin{equation*}
 c(f)c^+(g)-qc^+(g)c(f)=\left<f,g\right>\cdot 1.
\end{equation*}
In the $q=0$ case, these are the creation and annihilation operators on the full
Fock space.

It was shown by Voiculescu \cite{Voiculescu} that for a Hilbert
space of dimension $d\in \{1,2,\ldots,\infty\}$, the Hermitian parts of the
creation operators on the free Fock space
generate von Neumann algebras isomorphic to the free group
factor on $d$ generators. Thus, we can view the algebras $\Gamma_q(\mathcal
H):=\{c(f)+c^+(f): f\in\mathcal H\}''$ as $q$-deformations of the free group
factors. 

Various factoriality theorems have been proven for these
algebras. First, Bo{\.z}ejko and Speicher \cite{BS94} showed that these are
factors when $\dim \mathcal H$ is infinite. {\'S}niady \cite{Sniady} subsequently
showed that $\Gamma_q(\mathcal H)$ is a factor for $\dim\mathcal H$ sufficiently
large but finite. Ricard \cite{Ricard} showed that in fact $\Gamma_q(\mathcal
H)$ is a factor for $\dim \mathcal H\ge 2$.

More general deformations of the free group factors have also been considered.
For $\mathcal H$ a Hilbert space with basis $\{e_i\}_{i\in I}$, Bo{\.z}ejko and
Speicher \cite{BS94} constructed a solution to the $q_{ij}$-relations 
\begin{equation*}
 c(e_i)c^+(e_j)-q_{ij}c^+(e_j)c(e_i)=\delta_{ij},
\end{equation*}
for $q_{ij}\in [-1,1]$ as well as a further generalization of the relations arising from a contraction
$T\in\mathcal B(\mathcal H)$ satisfying the braid relation (or Yang-Baxter
relation) given by
\begin{equation*}
 (1\otimes T)(T\otimes 1)(1\otimes T)=(T\otimes 1)(1\otimes T)(T\otimes 1).
\end{equation*}

Kr{\.o}lak \cite{Krolak} proved that if $\|T\|<1$, which in the $q_{ij}$ case
corresponds to the condition $\sup\{|q_{ij}|:i,j\in I\}<1$, the resulting 
von Neumann algebra is a factor for $\dim \mathcal H$ sufficiently large.

In another direction, Liguouri and Mintchev \cite{LM} and  
Bo{\.z}ejko, Lytvynov, and Wysoczanski \cite{BLW} have considered 
creation and annihilation operators on a Fock space arising from a continuous
commutation relation associated with a Hermitian function
$Q$ from $\mathbb R^j\times \mathbb R^j$ (or some more general space) to the unit circle. This construction
also involves additional commutation relations on the creation operators, and includes the
anyons as a special case.

Here we will consider a continuous $Q$-commutation relation arising from
a function taking values in $(-1,1)$. Before we state the problem more explicitly,
we introduce some notations which will be used throughout the paper.

\begin{notation}
Let $U$ be a nonempty open subset of $\mathbb R^j$ for some integer $j\ge 1$. 
We also fix $Q\in C(U\times U)$, the space of continuous functions on $U\times U$.
Further assume that
$q:=\sup\{|Q(x,y)|: x,y\in U\} < 1$ and that $Q$ is a symmetric function, that is $Q(x,y) =
Q(y,x)$. Also define $\mathcal H=L^2(U)$.
\end{notation}

For points $x,y\in U$, we wish to consider, at least
heuristically, infinitesimal creation and annihilation operators on a 
$Q$-twisted Fock space satisfying the $Q$-commutation relation
\begin{equation}
 a(x)a^+(y)-Q(x,y)a^+(y)a(x) = \delta(x-y) \cdot 1, \label{eqn:qcr}
\end{equation}
where $\delta$ is the usual Dirac $\delta$, whence
\begin{equation*}
 \int\int \delta(x-y)f(x,y)\ dx\ dy = \int f(y,y)\ dy.
\end{equation*}
Rigorously, this relation should be understood
as a statement about operator-valued distributions, which makes sense upon smearing
with a test function and considering the resulting quadratic forms. The meaning
will be explained further in Section \ref{sec:fock}.

The operator-valued distributions $a^+(x)$ and $a(x)$ will give rise to 
creation and annihilation operators $a(f)$ and $a^+(f)$ on a $Q$-deformed
Fock space $\mathcal F_Q(\mathcal H)$. We will use these to define a $Q$-deformed
field operator $w(f)=a(f)+a^+(f)$ and the von Neumann algebra $\Gamma_Q(\mathcal H)$
generated by operators of this type.

This paper has four sections, not including this introduction. 
Section \ref{sec:fock} will present the construction of a deformed $Q$-Fock 
space with creation and annihilation operators realizing the $Q$-commutation 
relation. In Section \ref{sec:vNa}, we 
will discuss basic properties of the von Neumann algebra
generated by the field operators on this Fock space. In Section \ref{sec:discretization}, we will
show that the field operators arise as a limit in distribution of operators on discrete $q_{ij}$-Fock
spaces considered by Bo{\.z}ejko and Speicher in \cite{BS94}. In
Section \ref{sec:factoriality}, we will show that the von Neumann algebra generated by these operators is a factor.
\subsection*{Acknowledgments}
While working on this paper, the author was supported in part by a National Science Foundation (NSF) Graduate Research Fellowship. He was also supported in part by funds from NSF grant DMS-1001881. The author also benefited from attending the program ``Bialgebras in Free Probability'' at the Erwin Schr{\"o}dinger Institute in the Spring of 2011. His travel was supported by NSF grant DMS-1101630.

The author would like to thank Dan-Virgil Voiculescu for many enlightening conversations, Michael Hartglass for suggesting simplifications to the proof of Proposition \ref{prop:contractions}, and an anonymous referee for offering several corrections and for pointing out that several results are special cases of results of Krolak \cite{Krolak, Krolak2000}.
\section{The $Q$-Fock space}\label{sec:fock}
 We will construct our $Q$-Fock space by defining a deformed inner product on 
the algebraic Fock space. Fix $n$ and define for $1\le i\le n-1$ the operator $T_i^{(n)}$ on $\mathcal H^{\otimes n}$ by
\begin{equation*}
 T_i^{(n)}f(x_1,\ldots, x_n) = Q(x_i, x_{i+1})f(x_1,\ldots, x_{i-1}, x_{i+1}, x_i, x_{i+2}, \ldots, x_n).
\end{equation*}
Evidently $T_i$ is self-adjoint and bounded with norm at most 
$q:=\sup_{x,y}|Q(x,y)| < 1$. It is easily verified that
\begin{equation}\label{eqn:generalizedbraid}
T_i^{(n)}T_{j}^{(n)}= T_{j}^{(n)}T_i^{(n)}\ \text{for $|i-j|\ge 2$} \quad\text{and}\quad T_i^{(n)}T_{i+1}^{(n)}T_i^{(n)}=T_{i+1}^{(n)}T_i^{(n)}T_{i+1}^{(n)}.
\end{equation}
These relations are known as the Yang-Baxter relations, or sometimes the braid relations.
Now let $S_n$ denote the symmetric group on $n$ elements and
for $i=1,\ldots, n-1$ let $\pi_i$ be the permutation transposing $i$ and $i+1$ and
fixing all other elements. We define the map $\phi_n$ first on the $\pi_i$ by 
$\phi_n(\pi_i)=T_i^{(n)}$ and then on all of $S_n$ by quasi-multiplicative extension. This
means that if $\pi=\pi_{i_1}\cdots \pi_{i_k}$ is a decomposition of $\pi$ into a minimal
number of the $\pi_i$ then we define
\begin{equation*}
 \phi_n(\pi) = \phi_n(\pi_{i_1})\cdots \phi_n(\pi_{i_k})= T_{i_1}^{(n)}\cdots T_{i_k}^{(n)}.
\end{equation*}
That this definition does not depend on our choice of minimal length decompositions for
$\pi$ is a consequence of the fact that the $T_i^{(n)}$ satisfy \eqref{eqn:generalizedbraid}.
It follows from this definition that $\phi_n(\sigma_1\sigma_2)=\phi_n(\sigma_1)\phi_n(\sigma_2)$ 
whenever $|\sigma_1|+|\sigma_2|=|\sigma_1\sigma_2|$. Here $|\sigma_k|$ denotes
the number of inversions of the permutation $\sigma_k$. That is,
\begin{equation*}
 |\sigma_k|= \left|\left\{(i,j): 1\le i< j \le n, \sigma_k(i)> \sigma_k(j)\right\}\right|.
\end{equation*}
Equivalently, $|\sigma_k|$ is the length of the shortest word for $\sigma_k$ as a product
of the fundamental transpositions.

We now define the operator $P_Q^{(n)}\in\mathcal B(\mathcal H^{\otimes n})$ by 
\begin{equation*}
 P_Q^{(n)}\ = \ \sum_{\sigma\in S_n}\phi_n(\sigma).
\end{equation*}
By Theorem 2.3 of \cite{BS94}, the operator $P_Q^{(n)}$ is strictly positive.

Let $\mathcal F_{\text{alg}}(\mathcal H)$ be the algebraic Fock space on $\mathcal H$,
\begin{equation*}
 \mathcal F_{\text{alg}}(\mathcal H):=\bigoplus_{n=0}^\infty \mathcal{H}^{\otimes n},
\end{equation*}
where the direct sum is the algebraic direct sum, so that only finite sums are permitted. Here $\mathcal{H}^{\otimes 0}$ 
is a one-dimensional vector space generated by a distinguished unit vector $\Omega$,
which we call the vacuum vector.

The $Q$-inner products on the $\mathcal{H}^{\otimes n}$ naturally define a $Q$-inner product
on $\mathcal F_{\text{alg}}(\mathcal H)$ by sesquilinear extension of
\begin{equation*}
 \left<f,g\right>_Q=
\begin{cases}
\left<f, P_Q^{(n)}g\right>_0,&m=n,\\
0,&m\ne n,
\end{cases}
\end{equation*}
for $f\in \mathcal H^{\otimes n}$ and $g\in \mathcal H^{\otimes m}$. Here, $\left<\cdot,\cdot\right>_0$ denotes
the usual inner product on $\mathcal H^{\otimes n}$.
We now define the $Q$-Fock space $\mathcal F_Q(\mathcal H)$ as the completion of $\mathcal F_{\text{alg}}(\mathcal H)$
with respect to the $Q$-inner product. 

We are now almost ready to introduce the $Q$-creation and annihilation operators.
We will define these in terms of the free creation and annihilation operators.
For $f\in \mathcal H$, we define the free creation operator $l^+(h)$ on $\mathcal F_{\text{alg}}(\mathcal H)$ by 
\begin{equation*}
 l^+(h) f\ =\ h\otimes f
\end{equation*}
for $f\in \mathcal H^{\otimes n}$, where we adopt the convention for the $n=0$
case that $h\otimes \Omega = h$. We define $l(h)$ to be the free annihilation 
operator, given by
\begin{equation*}
 (l(h)f)(x_1,\ldots, x_{n-1}):=\int_{U} \overline{h(y)}f(y, x_1,\ldots, x_{n-1})\ dy.
\end{equation*}

We now define for $h\in\mathcal H$ the $Q$-creation 
operator $a^+(h)$ and the $Q$-annihilation operator $a(h)$ by 
\begin{equation*}
 a^+(h):= l^+(h)\quad\text{and}\quad a(h):= l(h)R_Q^{(n)}
\end{equation*}
on $\mathcal H^{\otimes n}$ for $n>0$, where
\begin{equation}
 R_Q^{(n)}:=1+T_1^{(n)}+T_1^{(n)}T_2^{(n)}+\cdots+T_1^{(n)}\cdots T_{n-2}^{(n)}T_{n-1}^{(n)}.\label{eqn:rdef}
\end{equation}
By writing each permutation $\sigma\in S_n$ as the product of an element of $S_1\times S_{n-1}$
and the minimal length representative of the coset of $\sigma$, we can
show that
\begin{equation}\label{eqn:oprelation}
 P_Q^{(n+1)} = (1\otimes P_Q^{(n)})R_Q^{(n+1)}.
\end{equation}
One can analogously define $Q$-deformed right creation and annihilation operators.
In general, we will state our results in terms of the left side versions, but
analogous results hold on the right side with the same proofs, and we will occasionally need
to make use of these analogs.

It was pointed out to us by an anonymous referee that the following is actually a special case of Theorem 3.1 of \cite{Krolak2000}.
\begin{proposition}
 For $h\in \mathcal H$, the operators $a(h)$ and $a^+(h)$ are adjoints
with respect to the $Q$-norm. Furthermore, for $h\in \mathcal H$, 
\begin{equation*}
 \|a^+(h)\|\le \|h\|\frac{1}{\sqrt{1-q}}.
\end{equation*}
In particular, $a^+(h)$ and $a(h)$ extend to bounded operators on $\mathcal F_Q(\mathcal H)$.
\end{proposition}
\begin{proof}
The proof of this theorem is very similar to that of Theorem 3.1 in \cite{BS94}.
We will first show that $a(h)$ and $a^+(h)$ are adjoints with respect to
the $Q$ inner product. The definitions imply that
\begin{equation*}
 l^+(h)T_i^{(n)}=T_{i+1}^{(n+1)}l^+(h),
\end{equation*}
whence it follows that 
\begin{equation*}
 l^+(h)P_Q^{(n)}=(1\otimes P_Q^{(n)})l^+(h)\quad \text{and} \quad P_Q^{(n)} l(h)=l(h)(1\otimes P_Q^{(n)}).
\end{equation*}
By applying \eqref{eqn:oprelation}, for $f\in \mathcal H^{\otimes n}$ 
\begin{align*}
 \left<a^+(h)f,g\right>_Q& = \left<a^+(h)f, P_Q^{(n+1)}g\right>_0\\
&=\left<f, l(h)P_Q^{(n+1)}g\right>_0\\
&=\left<f, l(h)\left(1\otimes P_Q^{(n)}\right)R_Q^{(n+1)}g\right>_0\\
&=\left<f, P_Q^{(n)}l(h)R_Q^{(n+1)}g\right>_0\\
&=\left<f, P_Q^{(n)}a(h)g\right>_0\\
&=\left<f, a(h)g\right>_Q.
\end{align*}
This proves that $a(h)$ and $a^+(h)$ are adjoints with respect to the $Q$-inner product.

We now prove the bound on $\|a^+(h)\|$. Since $\|T_i^{(n)}\|\le q$ for each $i$, 
\begin{equation*}
 \left\|R_Q^{(n)}\right\|\le 1+q+q^2+\cdots + q^{n-1}\le \frac{1}{1-q}.
\end{equation*}
Thus,
\begin{align*}
 P_Q^{(n+1)}P_Q^{(n+1)} &= P_Q^{(n+1)}\left(P_Q^{(n+1)}\right)^*\\ 
&=\left(1\otimes P_Q^{(n)}\right)R_Q^{(n+1)}\left(R_Q^{(n+1)}\right)^*\left(1\otimes P_Q^{(n)}\right)\\
&\le \frac{1}{(1-q)^2}\left(1\otimes P_Q^{(n)}\right)\left(1\otimes P_Q^{(n)}\right).
\end{align*}
Since $1\otimes P_Q^{(n)}$ and $P_Q^{(n+1)}$ are positive operators, it follows that
\begin{equation*}
 P_Q^{(n+1)}\le \frac{1}{1-q}\left(1\otimes P_Q^{(n)}\right).
\end{equation*}
Therefore, for $f\in \mathcal H^{\otimes n}$,
\begin{align*}
\left\|a^+(h)f\right\|^2 &= \left<a^+(h)f,a^+(h)f\right>_Q\\
&=  \left<h\otimes f,h\otimes f\right>_Q\\
&=  \left<h\otimes f, P_Q^{(n+1)}(h\otimes f)\right>_0\\
&\le \frac{1}{1-q} \left<h\otimes f, 1\otimes P_Q^{(n)}(h\otimes f)\right>_0\\
&\le \frac{1}{1-q} \left<h, h\right>_Q\left<f, P_Q^{(n)} f\right>_0\\
&\le \frac{1}{1-q} \left<h, h\right>\left<f, f\right>_Q\\
&\le \frac{1}{1-q} \|h\|^2 \|f\|_Q^2.
\end{align*}
\end{proof}
We can represent an element $f$ of the Fock space $\mathcal F_Q(\mathcal H)$ as
a sequence of functions $(f^{(0)}, f^{(1)}, \ldots)$, with $f^{(n)}\in \mathcal H^{\otimes n}$
and 
\begin{equation*}
 \sum_{n=0}^\infty \left\|f^{(n)}\right\|_Q^2 < \infty.
\end{equation*}
 
We are now ready to define the operator-valued distributions $a(x)$ and $a^+(x)$.
For $f\in \mathcal H^{\otimes n}$, we define these by
\begin{align*}
 [a(x)f](x_1,\ldots, x_{n-1})&= \left(R_Q^{(n+1)}f^{(n+1)}\right)(x, x_1,\ldots, x_{n-1})\\
 [a^+(x)f](x_1,\ldots, x_{n+1})&= \delta(x-x_1)f^{(n-1)}(x_2,\ldots, x_{n+1}). 
\end{align*}
These definitions, of course, makes no sense as functions, but should be interpreted 
as distributions on $C_c^\infty(U)$. It is an immediate consequence of the definitions that
\begin{equation*}
 a(h)=\int_U \overline{h(x)}a(x)\ dx\quad\text{and}\quad  a^+(h)=\int_U h(x)a^+(x)\ dx,
\end{equation*}
for functions $h\in C_c^\infty(U)$. These relations are understood rigorously in terms of
the corresponding quadratic forms. That is, for $f\in \mathcal H^{\otimes n}$
and $g\in\mathcal H^{\otimes (n-1)}$,
\begin{equation*}
\begin{split}
  \left<f, a^+(h)g\right>_Q&= \int_U h(x)\left< f, a^+(x)g\right>_Q\ dx\\
&=\int_U h(x)\int_{U^n}\left(P_Q^{(n)}f\right)(x_1,\ldots, x_{n})\delta(x-x_1)\overline{g(x_2,\ldots, x_n)} dx_1\ldots dx_n dx,
\end{split}
\end{equation*}
and similarly for $a(h)$:
\begin{equation*}
 \begin{split}
  \left<g, a(h)f\right>_Q&= \int_U h(x)\left< g, a(x)f\right>_Q\ dx\\
&=\int_U h(x)\int_{U^{n-1}}\left(P_Q^{(n-1)}g\right)(x_1,\ldots, x_{n-1})\overline{R_Q^{(n)}f(x, x_1,\ldots, x_{n-1})} dx_1\ldots dx_{n-1} dx.
\end{split}
\end{equation*}
It now follows from a simple computation that these operator-valued distributions satisfy the 
$Q$-commutation relations \eqref{eqn:qcr}. 
\section{The $Q$-deformed free group von Neumann algebras}\label{sec:vNa}
We now define the main operators of interest, the field operators $w(h)$
by
\begin{equation*}
 w(h):= a^+(h)+a(h)\quad\text{for}\quad h\in\mathcal H.
\end{equation*}
This allows us to define the $Q$-deformed free group von Neumann algebra by
\begin{equation*}
 \Gamma_Q(\mathcal H):=\left\{w(h):h\in \mathcal H\right\}''.
\end{equation*}

Before proving anything about these algebras, we will need some additional notation.
We will sometimes let $a^{-}(h)$ denote $a(h)$ so that we can write $a^{v}(h)$ for 
$v\in \{-,+\}$ to denote either the annihilation or creation operator. 

Given a finite ordered set $S$, we will denote the set of pairings of $S$ by
$P(S)$. That is, $P(S)=\emptyset$ if $S$ has odd cardinality, and if $|S| = 2p$ then
\begin{equation*}
P(S)=\{\{(a_1,z_1),\ldots, (a_p,z_p)\}| a_1<z_1,\ldots, a_p<z_p, \{a_1,\ldots, a_p, z_1,\ldots, z_p\} = S\}.
\end{equation*}

We will denote by $I(\mathcal V)$ the set of crossings of a 
pairing $\mathcal V$, that is, for $\mathcal V=\{(a_1,z_1),\ldots, (a_p,z_p)\}$, 
\begin{equation*}
 I(\mathcal V) = \{(k,l)\in \{1,\ldots, r\}^2|a_k<a_l<z_k<z_l\},
\end{equation*}
 where the inequalities are in the ordering given on $S$.

For a pairing $\mathcal V\in P(S)$ for $S\subset \{1,\ldots, n\}$, we
define a function $Q_{\mathcal V}^n$ on $U^n$ by
\begin{equation*}
 Q_{\mathcal V}^n(\mathbf{x})=\prod_{(k,l)\in I(\mathcal V)}Q(x_{a_k},x_{a_l}).
\end{equation*}
We will simplify notation by writing
\begin{equation*}
 \delta^n_{\mathcal V}(\mathbf{x})=\prod_{(a,z)\in \mathcal
V}\delta(x_a-x_z).
\end{equation*}
Note that the $\delta$ on the right side is the Dirac delta.
\begin{proposition}\label{prop:vac}
Let $f_1,\ldots, f_n\in \mathcal H$ and denote by $S$ the set $\{1,\ldots, n\}$. For $v_1,\ldots, v_n\in\{-,+\}$
\begin{equation*}
 \left<a^{v_n}(f_n)\cdots a^{v_1}(f_1)\Omega, \Omega\right> = \sum_{\mathcal V\in P(S)}D_{\mathbf{v},\mathcal V}
\int\cdots \int f_n(x_n)\cdots f_1(x_1) Q_{\mathcal V}^n(\mathbf{x})\delta^n_{\mathcal V}(\mathbf{x}) dx_1\ldots dx_n,
\end{equation*}
where if $n=2p$, $D_{\mathbf{v},\mathcal V}$ is defined by
\begin{equation*}
 D_{\mathbf{v},\mathcal V}=\prod_{k=1}^p \delta_{v_{a_k},-}\cdot \delta_{v_{z_k},+}.
\end{equation*}
In particular, $\left<w(f_n)\cdots w(f_1)\Omega, \Omega\right> = 0$ when $n$ is odd.
\end{proposition}
\begin{proof}
The proof of is by induction on $N:=\left|\{(j,k): j<k, v_j =
+, v_k = -\}\right|$.  The claim is easily seen to be true in the case $N=0$, so we proceed to
assume that $N>0$ and that the claim holds for $N-1$. We will assume that $f_1,\ldots, f_n$ 
lie in the dense subspace $C_c^\infty(U)$ of $\mathcal H$ and then use the $Q$-commutation relation
\eqref{eqn:qcr}. Since $N>0$, we can choose $j$ minimal to satisfy $v_j =+$ and $v_{j+1} =-$. Now
applying \eqref{eqn:qcr},
\begin{align}
a^{v_n}(x_n)\cdots a^{v_1}(x_1) &= a^{v_n}(x_n)\cdots
a^{v_{j+2}}(x_{j+2})a(x_{j+1})a^+(x_{j})a^{v_{j-1}}(x_{j-1}) \cdots
a^{v_1}(x_1) \nonumber\\
&= a^{v_n}(x_n)\cdots a^{v_{j+2}}(x_{j+2})(Q(x_j,
x_{j+1})a^+(x_{j})a(x_{j+1})+\delta(x_j,x_{j+1}))a^{v_{j-1}}(x_{j-1}) \cdots
a^{v_1}(x_1)\nonumber\\
&= Q(x_j, x_{j+1})a^{v_n}(x_n)\cdots
a^{v_{j+2}}(x_{j+2})a^+(x_{j})a(x_{j+1})a^{v_{j-1}}(x_{j-1}) \cdots
a^{v_1}(x_1)\nonumber \\ &\quad+ \delta(x_j-x_{j+1})a^{v_n}(x_n)\cdots
a^{v_{j+2}}(x_{j+2})a^{v_{j-1}}(x_{j-1}) \cdots a^{v_1}(x_1)
\label{eqn:nogoodname}
\end{align}

We now consider the terms in the last line of \eqref{eqn:nogoodname} separately, denoting them by $X_1$
and $X_2$. For compactness of notation, we define $S'=\{1,\ldots, j-1,
j+1, j, j+2,\ldots, n\}$ (as an ordered set) and $\hat S=\{1,\ldots, j-1,
j+2,\ldots, n\}$ and also write $\mathbf{f}(\mathbf{x})$ for the product
$f_n(x_n)\cdots f_1(x_1)$.

For the first term we have by the inductive hypothesis,
\begin{align*}
\int \cdots \int  \mathbf{f}(\mathbf{x}) \left<X_1\Omega,\Omega\right>\ dx_1\ldots dx_n &=
\sum_{\mathcal V\in P(S')}D_{\mathbf{v},\mathcal
V}\int\cdots \int \mathbf{f}(\mathbf{x})Q(x_j, x_{j+1})Q^n_{\mathcal V}(\mathbf{x})  \delta^n_{\mathcal V}(\mathbf{x}) dx_1\ldots dx_n\\
&= \sum_{\substack{\mathcal V\in P(S)\\ (j,j+1)\not\in \mathcal V}}D_{\mathbf{v},\mathcal V}
\int\cdots \int \mathbf{f}(\mathbf{x})Q^n_{\mathcal
V}(\mathbf{x})\delta^n_{\mathcal V}(\mathbf{x}) \ dx_1\ldots dx_n,
\end{align*}

For the second term, 
\begin{align*}
\int \cdots \int \mathbf{f}(\mathbf{x})\left<X_2\Omega,\Omega\right>\ dx_1\ldots dx_n &=
\sum_{\mathcal V\in P(\hat S)}D_{\mathbf{v},\mathcal V} \int\cdots \int \delta(x_j-x_{j+1})
\mathbf{f}(\mathbf{x})
Q^n_{\mathcal V}(\mathbf{x})\delta^n_{\mathcal V}(\mathbf{x}) \ dx_1\ldots dx_n\\
&= \sum_{\substack{\mathcal V\in P(S)\\(j,j+1)\in \mathcal V}}D_{\mathbf{v},\mathcal V} 
\int\cdots \int \mathbf{f}(\mathbf{x})
Q^n_{\mathcal V}(\mathbf{x}) \delta^n_{\mathcal V}(\mathbf{x})  \ dx_1\ldots dx_n.
\end{align*}

The proposition now follows just by adding the results of the two computations
just completed.
\end{proof}
\begin{corollary}\label{cor:vac}
 Let $f_1,\ldots, f_n$ and $S$ be as in Proposition \ref{prop:vac}. Then
\begin{equation*}
 \left<w(f_n)\cdots w(f_1)\Omega, \Omega\right> = \sum_{\mathcal V\in P(S)}
\int\cdots \int f_n(x_n)\cdots f_1(x_1) Q_{\mathcal V}^n(\mathbf{x})\delta^n_{\mathcal V}(\mathbf{x}) dx_1\ldots dx_n,
\end{equation*}
\end{corollary}
\begin{proof}
 Sum the formula of Proposition \ref{prop:vac} over all choices of $v_1,\ldots, v_n$.
\end{proof}

\begin{corollary}
 The vacuum state on $\Gamma_Q(\mathcal H)$ is a trace.
\end{corollary}
\begin{proof}
 The formula in Corollary \ref{cor:vac} is invariant under cyclic permutations
of the $w(f_i)$.
\end{proof}

\begin{proposition}
 The vacuum vector $\Omega\in \mathcal F_Q(\mathcal H)$ is cyclic and separating 
for $\Gamma_Q(\mathcal H)$. 
\end{proposition}
\begin{proof}
We first show that $\Omega$ is cyclic. It will suffice to show that an
arbitrary $f\in \mathcal H^{\otimes n}$ is in the closure of $\Gamma_Q(\mathcal H)\Omega$.
The proof is by induction on $n$. The cases of $n=0$ and $n=1$ are obvious,
so we assume $n>1$ and $f\in L^2(U^n)$. If $\epsilon > 0$, we can choose
$(f_{ij})\in \mathcal H$ for $i=1,\ldots n$ and $j=1,\ldots, r$ such that
\begin{equation*} 
\left\|f-\sum_{j=1}^r f_{1j}\otimes \cdots\otimes f_{nj}\right\| < \epsilon/2.
\end{equation*}
But then
\begin{equation*} 
f-\sum_{j=1}^r w(f_{1j}) \cdots w(f_{nj})\Omega = \left(f-\sum_{j=1}^r f_{1j}\otimes \cdots\otimes f_{nj}\right) + g,
\end{equation*}
for some $g\in \bigoplus_{k=1}^{n-1}\mathcal H^{\otimes n}$. The claim now follows by
applying the inductive hypothesis to $g$.

To show that $\Omega$ is separating for $\Gamma_Q(\mathcal H)$, it will suffice
to show that $\Omega$ is cyclic for $\Gamma_Q(\mathcal H)'$. We define the 
anti-linear conjugation operator $J:\mathcal F_Q(\mathcal H)\to \mathcal F_Q(\mathcal H)$
by $JX\Omega=X^*\Omega$ for $X\in\Gamma_Q(\mathcal H)$. This operator is well-defined
because by the tracial property $\|X\Omega\|=\|X^*\Omega\|$. Since $J\Gamma_Q(\mathcal H)J$
commutes with $\Gamma_Q(\mathcal H)$, and $\Omega$ is seen to be cyclic for $J\Gamma_Q(\mathcal H)J$
in the same way as for $\Gamma_Q(\mathcal H)$, the claim follows.
\end{proof}

\section{The discretization lemma}\label{sec:discretization}
We will now show that the creation and annihilation operators $a^+(h)$ and $a(h)$ can be realized
as a limit in distribution of operators on a discrete Fock space arising from
the discrete commutation relation as considered in \cite{BS94}. Fix $\epsilon$ and let 
$U_\epsilon:=U\cap \epsilon \mathbb Z^j$. We let $\mathcal H_{\epsilon}$ be a real
Hilbert space with orthonormal basis $\{e_x: x\in U_{\epsilon}\}$. For $x,y\in U_\epsilon$,
we define $q_{xy} = Q(x, y)$.

Bozejko and Speicher showed \cite{BS94} that there is a $q_{xy}$-Fock space 
on $\mathcal H_{\epsilon}$ with vacuum vector $\Omega_\epsilon$,
creation operators $a^+_{\epsilon}(f)$ and annihilation
operators $a_{\epsilon}(e)$ for $e\in \mathcal H$ satisfying the discrete $q_{xy}$-commutation
relation
\begin{equation*}
 a_{\epsilon}(e_x)a_{\epsilon}^+(e_y) - q_{ij}
a_{\epsilon}^+(e_y)c_{\epsilon}(e_x) = \delta_{xy}\cdot 1.
\end{equation*}
The creation operator $a_{\epsilon}^+(e_x)$ and the annihilation operator $a_{\epsilon}(e_x)$
are adjoints with respect to the deformed inner product on the Fock space.
We will denote this Fock space by $\mathcal F_{Q,\epsilon}(\mathcal H_\epsilon)$, its vaccum vector by $\Omega_\epsilon$, and its inner product by $\left<\cdot,\cdot\right>_{Q,\epsilon}$.

Now define $a_\epsilon(f)$ and $a^+_\epsilon(f)$ by
\begin{equation*}
a_\epsilon(f): =\epsilon^{j/2}\sum_{x\in U_\epsilon} f(x) a_\epsilon(e_{x})
\quad\text{and}\quad
a^+_\epsilon(f): =\epsilon^{j/2}\sum_{x\in U_\epsilon} f(x) a_\epsilon^+(e_{x}).
\end{equation*}
Evidently, $(a_\epsilon(f))^* = a^+_\epsilon(\overline{f})$.

To simplify notation, we define for a pairing $\mathcal V$,
\begin{equation*}
D^n_{\mathcal V}(\mathbf{x})=\prod_{(a,z)\in \mathcal V}\delta_{x_a,x_z},
\end{equation*}
where the $\delta_{x_a,x_z}$ on the right side is a Kronecker delta.

\begin{lemma}\label{lem:discretization}
The family $\{a_\epsilon(f):f\in C_c^\infty(U)\}$ converges 
in joint $*$-distribution as $\epsilon\to 0$ to the family 
$\{a(f):f\in C_c^\infty(U)\}$ introduced in Section \ref{sec:fock}
where all of the distributions are with respect to the respective vacuum states.
\end{lemma}
\begin{proof}
We will use the fact, as shown by Bo{\.z}ejko and Speicher in \cite{BS94}, that for
$v_1,\ldots, v_n\in \{+,-\}$,
\begin{equation*}
\left<a^{v_n}_\epsilon(x_n)\cdots a^{v_1}_\epsilon(x_1)\Omega_\epsilon, \Omega_\epsilon\right>
= \sum_{\mathcal V\in P(S)}D_{\mathbf{v},\mathcal V} D_{\mathcal V}^n(\mathbf{x}) \prod_{(k,l)\in I(\mathcal V)} q_{x_{a_k}, x_{a_l}},
\end{equation*}
where $S=\{1,\ldots, n\}$ and $D_{\mathbf{v},\mathcal V}$ is as in Proposition \ref{prop:vac}. Again writing $\mathbf{f}(\mathbf{x})$ for the product
$f_n(x_n)\cdots f_1(x_1)$, we have that
\begin{align*}
\lim_{\epsilon\to 0} \left<a^{v_n}_\epsilon(f_n)\cdots a^{v_1}_\epsilon(f_1)\Omega_\epsilon, \Omega_\epsilon\right>
&= \lim_{\epsilon\to 0} \epsilon^{jn/2}\sum_{\mathbf{x}\in U_\epsilon^n}
\left< \mathbf{f}(\mathbf{x}) a^{v_n}_{\epsilon}(e_{x_n})\cdots a^{v_1}_{\epsilon}(e_{x_1})\Omega_\epsilon, \Omega_\epsilon\right>\\
&= \lim_{\epsilon\to 0} \epsilon^{jn/2}\sum_{\mathbf{x}\in U_\epsilon^n}
\mathbf{f}(\mathbf{x})\sum_{\mathcal V\in P(S)}D_{\mathbf v,\mathcal V}D_{\mathcal V}^n(\mathbf{x})\prod_{(k,l)\in I(\mathcal V)}
q_{x_{a_k}, x_{a_l}}\\
&= \int\cdots\int \mathbf{f}(\mathbf{x})\sum_{\mathcal V\in P(S)} D_{\mathbf v,\mathcal V}
 D^n_{\mathcal V}(\mathbf{x})Q_{\mathcal V}^n(\mathbf{x}) \ dx_1\ldots dx_n\\
&=  \left<a^{v_n}(f_n)\cdots a^{v_1}(f_1)\Omega, \Omega\right>.
\end{align*}
\end{proof}

We conclude this section by noting that the inner product on 
$\mathcal F_{Q,\epsilon}(\mathcal H_\epsilon)$ is defined using positive operators
$P_{Q,\epsilon}^{(n)}$ on $\mathcal H_{\epsilon}^{\otimes n}$ such that 
\begin{equation*}
 \left<\xi,\eta\right>_{Q,\epsilon} = \left<\xi, P_{Q,\epsilon}^{(n)}\eta\right>_{0,\epsilon},
\end{equation*}
for $\xi,\eta\in \mathcal H^{\otimes n}$, where $\left<\cdot,\cdot\right>_{0,\epsilon}$ denotes the inner product
of the Free fock space on $\mathcal H_\epsilon$. Since we have assumed that $\sup_{x,y}|Q(x,y)|< 1$, 
there is an operator $R_{Q,\epsilon}^{(n)}$ of norm at most
$(1-q)^{-1}$ such that $P_{Q,\epsilon}^{(n+1)} = \left(1\otimes P_{Q,\epsilon}^{(n)}\right) R_{Q,\epsilon}^{(n)}$.
One can use this to show that $P_{Q,\epsilon}^{(n+1)} \le (1-q)^{-1}(1\otimes P_{Q,\epsilon}^{(n)})$
for all $\epsilon$.

\section{The factoriality result}\label{sec:factoriality}
To state our main theorem, we will need to introduce the right field operator $w_r(f)$ for $f\in \mathcal H$. We define
\begin{equation*}
 w_r(f)=Jw(f)J,
\end{equation*}
where $J:\mathcal F_Q(\mathcal H)\to \mathcal F_Q(\mathcal H)$ is the canonical antilinear isometry defined
by $J(X\Omega)=X^*\Omega$. Equivalently, 
\begin{equation*}
 w_r(f)=a_r(f)+a_r^+(f),
\end{equation*}
where $a_r(f)$ and $a_r^+(f)$ are the right annihilation and right creation operators defined analogously
to the left annihilation and left creation operators.

\begin{theorem}\label{thm:operator}
Let $g_1,g_2,\ldots \in C_c^\infty (U)$ be
real-valued functions with $g_ig_j=0$ for $i\ne j$ and $\|g_i\|_2=1$. For each $d> 0$, define
\begin{equation*}
 N_d = \sum_{i=1}^d (w(g_i)-w_r(g_i))^2.
\end{equation*}
Then for $d$ sufficiently large, $\ker N_d = \mathbb C\Omega$ and $N_d >\epsilon 1$ on $\mathcal F_Q(\mathcal H)\ominus \mathbb C\Omega$ for some $\epsilon > 0$.
\end{theorem}
In view of a theorem of Connes \cite{Connes}, this theorem will have the following consequence. It was pointed out to us by an anonymous referee that this follows immediately from the main theorem of Krolak in \cite{Krolak}.
\begin{corollary}\label{cor:factor}
The von Neumann algebra $\Gamma_Q(\mathcal H)$ is a factor which 
does not have property $\Gamma$.
\end{corollary}
\begin{proof}
Choose $N_d$ large enough that $N_d>\epsilon1$ on the orthogonal complement of the vacuum subspace.
If $X\in \Gamma_Q(\mathcal H)\cap \Gamma_Q(\mathcal H)'$ 
then $\left(w(g_i)-w_r(g_i)\right)X=0$ for $i=1,\ldots,d$. Thus $X\Omega\in \ker N_d=\mathbb C\Omega$.
Since $\Omega$ is separating, $X\in\mathbb C$. Thus, $\Gamma_Q(\mathcal H)$ is a factor of Type $II_1$.
By Theorem 2.1 of \cite{Connes}, $\Gamma_Q(\mathcal H)$ does not have property $\Gamma$.
\end{proof}

Our method of proof of Theorem \ref{thm:operator} will be similar to that used by Krolak 
\cite{Krolak} and will require some estimates.
\begin{proposition}\label{prop:contractions}
For each $n$, define operators 
\begin{equation*}
\mathcal L_n: \mathcal H \otimes  \mathcal H^{\otimes (n-1)} \to \mathcal H^{\otimes (n-2)}
\quad\text{and}\quad
\mathcal R_n: \mathcal H^{\otimes (n-1)}\otimes \mathcal H \to \mathcal H^{\otimes (n-2)}
\end{equation*}
by
\begin{equation*}
 \mathcal L_n(h\otimes f) = l(h)f
\quad\text{and}\quad
 \mathcal R_n (f\otimes h) = r(h)f,
\end{equation*}
where  $l(f)$ and $r(f)$ are the free left and right annihilation operators, respectively
acting on $\mathcal H^{\otimes (n-1)}$ as a subspace of $\mathcal F_Q(\mathcal H)$.
Suppose that $g\in \mathcal H$ with $\|g\|=1$ and define $D$ on $\mathcal H^{\otimes n}$ by $D(f)= g\otimes f\otimes g$.
Then
\begin{equation*}
\left\|\mathcal L_{n+2}(T_2^{(n+2)}\cdots T_{n+1}^{(n+2)})D\right\|_Q \le q^n
\quad\text{and}\quad
\left\|\mathcal R_{n+2}(T_n^{(n+2)}\cdots T_1^{(n+2)})D\right\|_Q \le q^n.
\end{equation*}
\end{proposition}
\begin{proof}
We will prove the second statement, and the first can be proven analogously. Our approach is similar to that of Lemma 7 in \cite{Krolak}.  Namely, we will begin by showing that the operator $\mathcal R_{n+2}(T_n^{(n+2)}\cdots T_{1}^{(n+2)})D$
commutes with  $P_Q^{(n)}$, the operator used to define the $Q$-inner product in Section \ref{sec:fock}. For this,
it will suffice to show that $\mathcal R_{n+2}(T_n^{(n+2)}\cdots T_{1}^{(n+2)})D$ commutes with $\phi_n(\sigma)$ for each $\sigma\in S_n$, where
$\phi_n:S_{n}\to \mathcal H^{\otimes n}$ is as in Section \ref{sec:fock}. By quasimultiplicativity of $\phi_n$,
we can further assume that $\sigma$ is one of the fundamental transpositions $\pi_k$. Using the relation \eqref{eqn:generalizedbraid}, we have
\begin{equation*}
\begin{split}
 \mathcal R_{n+2}(T_n^{(n+2)}\cdots T_1^{(n+2)})D\phi_n(\pi_k) &= \mathcal R_{n+2}(T_n^{(n+2)}\cdots T_1^{(n+2)})DT_k^{(n)}\\
&= \mathcal R_{n+2}(T_n^{(n+2)}\cdots T_1^{(n+2)})T_{k+1}^{(n+2)}D\\
&= \mathcal R_{n+2}(T_n^{(n+2)}\cdots T_{k+2}^{(n+2)}T_{k+1}^{(n+2)}T_{k}^{(n+2)}T_{k+1}^{(n+2)}T_{k-1}^{(n+2)}\cdots T_1^{(n+2)})D\\
&= \mathcal R_{n+2}(T_n^{(n+2)}\cdots T_{k+2}^{(n+2)}T_{k}^{(n+2)}T_{k+1}^{(n+2)}T_{k}^{(n+2)}T_{k-1}^{(n+2)}\cdots T_1^{(n+2)})D\\
&= \mathcal R_{n+2}T_{k}^{(n+2)}(T_n^{(n+2)}\cdots T_1^{(n+2)})D\\
&=  T_{k}^{(n)}\mathcal R_{n+2}(T_n^{(n+2)}\cdots T_1^{(n+2)})D\\
&=  \phi_n(\pi_k)\mathcal R_{n+2}(T_n^{(n+2)}\cdots T_1^{(n+2)})D.
\end{split}
\end{equation*}
Therefore,  $\mathcal R_{n+2}(T_n^{(n+2)}\cdots T_1^{(n+2)})D$ commutes with $P_Q^{(n)}=\sum_{\sigma\in S_n}\phi_n(\sigma)$. In particular,
this means that
\begin{align*}
 \left\|\mathcal R_{n+2}(T_n^{(n+2)}\cdots T_1^{(n+2)})D\phi_n(\sigma)\right\|_Q&=\left\|\mathcal R_{n+2}(T_n^{(n+2)}\cdots T_1^{(n+2)})D\phi_n(\sigma)\right\|_0\\
&\le \left\|\mathcal R_{n+2}\right\|_0\left\|(T_n^{(n+2)}\cdots T_1^{(n+2)})\right\|_0\left\|D\phi_n(\sigma)\right\|_0\\
&\le 1\cdot q^n\cdot 1.
\end{align*}
In the last line, we have used the fact that $D$ is an isometry in the $0$-norm and $\mathcal R_{n+2}$ is a
contraction when restricted to the subspace $\mathcal H^{\otimes (n+1)}\otimes g$.
\end{proof}
The next lemma provides an analog to parts of Lemma 8 of \cite{Krolak}.
\begin{lemma}\label{lem:bounds}
There is a constant $C$, depending only on $Q$, such that all of the following estimates hold for any orthonormal vectors $h_1,\ldots, h_d\in\mathcal H$:
\begin{enumerate}
 \item 
 $\left\|\sum_{i=1}^d a^+(h_i)a_r^+(h_i)\right\|_Q \le C\sqrt{d}$ and 
 $\left\|\sum_{i=1}^d a(h_i)a_r(h_i)\right\|_Q  \le C\sqrt{d}$ \label{itm:lcrrcr}
 \item 
 $\left\|\sum_{i=1}^d a^+(h_i)a_r(h_i)\right\|_Q \le C\sqrt{d}$ and 
 $\left\|\sum_{i=1}^d a_r^+(h_i)a(h_i)\right\|_Q  \le C\sqrt{d}$ \label{itm:lcrran}
 \item
 $\left\|\sum_{i=1}^d a(h_i)a(h_i)\right\|_Q \le C\sqrt{d}$ and  
 $\left\|\sum_{i=1}^d a_r(h_i)a_r(h_i)\right\|_Q  \le C\sqrt{d}$ \label{itm:lanlan}
 \item 
 $\left\|\sum_{i=1}^d a^+(h_i)a(h_i)\right\|_Q \le C\sqrt{d}$ and 
 $\left\|\sum_{i=1}^d a_r^+(h_i)a_r(h_i)\right\|_Q  \le C\sqrt{d}$ \label{itm:lcrlan}
\end{enumerate}
\end{lemma}
\begin{proof}
We take $C=\frac{1}{1-q}$, which is large enough so that $P_Q^{(n+1)} \le C (1\otimes P_Q^{(n)})$
for all $n$ is as established in \eqref{eqn:rdef}.  In general, to prove that an operator
$X$ has norm at most $K$, it will be sufficient to prove that $\|Xf\|^2\le K^2\|f\|^2$ for all
of the form $f \in\mathcal H^{\otimes n}$ where $n\ge 0$ is arbitrary.
To prove the first bound in part \ref{itm:lcrrcr}, we have
\begin{equation*}
 \begin{split}
\left\|\sum_{i=1}^d a^+(h_i)a_r^+(h_i)f\right\|_Q^2 & = \left\|\sum_{i=1}^d h_i\otimes f \otimes h_i\right\|_Q^2\\
&= \left<\sum_{i=1}^d h_i\otimes f\otimes h_i, \sum_{j=1}^d P_Q^{(n+2)}h_j\otimes f\otimes h_j\right>_Q\\
&\le C^2\left<\sum_{i=1}^d h_i\otimes f\otimes h_i, \sum_{j=1}^d (1\otimes P_Q^{(n)}\otimes 1)h_j\otimes f\otimes h_j\right>_0\\
&= C^2\sum_{i,j=1}^d\left<h_i,h_j\right>\left<h_i,h_j\right>\left< f, P_Q^{(n)}f\right>_0\\
&= C^2\sum_{i=1}^d\left< f, P_Q^{(n)}f\right>_0\\
&= dC^2\|f\|_Q^2
 \end{split}
\end{equation*}
For the second bound in part \ref{itm:lcrrcr}, we have that
\begin{equation*}
 \left\|\sum_{i=1}^d a(h_i)a_r(h_i)\right\|_Q =  \left\|\left(\sum_{i=1}^d a_r^+(h_i)a^+(h_i)\right)^*\right\|_Q
 =  \left\|\sum_{i=1}^da^+_r(h_i) a^+(h_i)\right\|_Q\le C\sqrt{d},
\end{equation*}
where in the last line we have used the first bound  in part \ref{itm:lcrrcr}.

The proof of the first bound in  part \ref{itm:lcrran} is similar:
\begin{align*}
\left\|\sum_{i=1}^d a^+(h_i)a_r(h_i)f\right\|_Q^2&=\left\|\sum_{i=1}^d h_i\otimes a_r(h_i)f\right\|_Q^2\\
&=\sum_{i,j=1}^d\left<P_Q^{(n)} \left(h_i\otimes a_r(h_i)f\right), h_j\otimes a_r(h_j)f\right>_Q\\
&\le C\sum_{i,j=1}^d\left<(1\otimes P_Q^{(n-1)}) \left(h_i\otimes a_r(h_i)f\right), h_j\otimes a_r(h_j)f\right>_Q\\
&\le C\sum_{i=1}^d\left<P_Q^{(n-1)}a_r(h_i)f, a_r(h_i)f\right>_Q\\
&\le C\sum_{i=1}^d\left\|a_r(h_i)f\right\|_Q^2\\
&\le dC^2\|f\|_Q^2.
\end{align*}
The arguments used to prove the second inequality in part \ref{itm:lcrran} and all the remaining estimates are similar to those cases just completed.
\end{proof}
We will need one additional bound, which is analogous to  the last part of Lemma 8 of \cite{Krolak}.
\begin{proposition}\label{prop:ancr}
If $h_1,\ldots, h_d\in C_c^\infty(U)$ are such that $\|h_i\|_2=1$ and $h_ih_j=0$ for $i\ne 0$ then
there is a constant $C$, depending only on $Q$, such that
\begin{equation*}
\left\| \sum_{i=1}^d\left( a(h_i)a^+(h_i)-1\right) \right\|\le Cq\sqrt{d}
\quad\text{and}\quad
\left\| \sum_{i=1}^d\left( a_r(h_i)a_r^+(h_i)-1\right) \right\|\le Cq\sqrt{d}.
\end{equation*}
\end{proposition}
\begin{proof}
We will prove the first estimate; the proof of the second is analogous. It will suffice to show that for $f=\sum_{j\in J}f_{1j}\otimes\cdots \otimes f_{nj}$ with $f_{1j},\cdots, f_{nj}\in C_c^\infty(U)$,
\begin{equation*}
 \left\| \sum_{i=1}^d\left( a(h_i)a^+(h_i)-1\right)f \right\|_Q^2\le q^2C^2d\left\|f \right\|_Q^2.
\end{equation*}
To prove this result, we will make use of Lemma \ref{lem:discretization}, which implies that
in the notation of Section \ref{sec:discretization},
\begin{equation*}
 \left\| \sum_{i=1}^d\left( a(h_i)a^+(h_i)-1\right)f \right\|_Q^2=  
\lim_{\epsilon\to 0}\left\| \sum_{i=1}^d\left( a_\epsilon(h_i)a_\epsilon^+(h_i)-1\right)\sum_{j\in J}a_{\epsilon}^+(f_{1j})\cdots a_{\epsilon}^+(f_{nj})\Omega_\epsilon \right\|_{Q,\epsilon}^2,
\end{equation*}
We again choose $C=\frac{1}{1-q}$. For this choice of the constant, we have 
$P_{Q,\epsilon}^{(n)} \ge C(1\otimes P_{Q,\epsilon}^{(n-1)})$ and also 
$P_{Q,\epsilon}^{(n)} \ge C(P_{Q,\epsilon}^{(n-1)}\otimes 1)$.
We define $f_{\epsilon} = \sum_{j\in J}a_{\epsilon}^+(f_{1j})\cdots a_{\epsilon}^+(f_{nj})\Omega_\epsilon $, fix $\epsilon > 0$, and denote 
$\sum_{i=1}^d\left( a_\epsilon(h_i)a_\epsilon^+(h_i)-1\right)$ by $V_\epsilon$.
Applying the discrete commutation relations and rearranging terms,
\begin{align*}
 \left\| V_\epsilon f_{\epsilon} \right\|_{Q,\epsilon}
&=\left\| \sum_{i=1}^d\left( \left(\epsilon^j\sum_{x_1,x_2\in U_\epsilon}\overline{h_i}(x_1)h_i(x_2)a_\epsilon(e_{x_1})a_\epsilon^+(e_{x_2})\right)-1\right)f_{\epsilon} \right\|_{Q,\epsilon}\\
&=\left\| \left(\sum_{i=1}^d\sum_{x_1,x_2\in U_\epsilon}
\epsilon^j\overline{h_i}(x_1)h_i(x_2)\left(Q(x_1,x_2)a_\epsilon^+(e_{x_2})a_\epsilon(e_{x_1})+\delta_{x_1,x_2}\right)f_\epsilon\right)-df_\epsilon \right\|_{Q,\epsilon}\\
&\le \left\|\sum_{i=1}^d\sum_{\substack{x_1\in U_\epsilon\\ x_2\in U_\epsilon}}
\epsilon^j\overline{h_i}(x_1)h_i(x_2)Q(x_1,x_2)a_\epsilon^+(e_{x_2})a_\epsilon(e_{x_1})f_\epsilon\right\|_{Q,\epsilon} + \left\|\sum_{i=1}^d\left(-1+\epsilon^j\sum_{x\in U_\epsilon}|h_i(x)|^2\right)f_\epsilon \right\|_{Q,\epsilon}.
\end{align*}
Since $\|h_i\|^2 = 1$, the second term in the last line converges to $0$ as $\epsilon\to 0$, whence we need
only show that the first term has the needed bound in the limit. Denoting this term by $S_\epsilon$, we have
\begin{align*}
 S_\epsilon^2
&= \left\|\sum_{i=1}^d\sum_{x_1,x_2\in U_\epsilon}
\epsilon^j\overline{h_i}(x_1)h_i(x_2)Q(x_1,x_2)a_\epsilon^+(e_{x_2})a_\epsilon(e_{x_1})f_\epsilon\right\|_{Q,\epsilon}^2\\
&= \left\|\sum_{i=1}^d\sum_{x_1,x_2\in U_\epsilon}
\epsilon^j\overline{h_i}(x_1)h_i(x_2)Q(x_1,x_2)e_{x_2}\otimes a_\epsilon(e_{x_1})f_\epsilon\right\|_{Q,\epsilon}^2\\
&\le C\sum_{x_2\in U_\epsilon} \left\|\sum_{i=1}^d\sum_{x_1\in U_\epsilon}
\epsilon^j\overline{h_i}(x_1)h_i(x_2)Q(x_1,x_2)a_\epsilon(e_{x_1})f_\epsilon\right\|_{Q,\epsilon}^2.
\end{align*}
Here we have used the fact that $P_Q^{(n+1)} \le C (1\otimes P_Q^{(n)})$. To further simplify this bound,
we use the fact that the adjoint map is an isometry and then make use of our
choice of $C$ again:
\begin{align*}
S_\epsilon^2
&\le C\sum_{x_2\in U_\epsilon} \left\|\sum_{i=1}^d\sum_{x_1\in U_\epsilon}
\epsilon^j{h_i}(x_1)\overline{h_i}(x_2)Q(x_1,x_2)a_\epsilon^+(e_{x_1})\right\|_{Q,\epsilon}^2
\|f_\epsilon\|_{Q,\epsilon}^2 \\ 
&\le C\sum_{x_2\in U_\epsilon} \sup_{\|g_\epsilon\|_{Q,\epsilon}=1}\left\|\sum_{i=1}^d\sum_{x_1\in U_\epsilon}
\epsilon^j{h_i}(x_1)\overline{h_i}(x_2)Q(x_1,x_2)e_{x_1}\otimes g_\epsilon \right\|_{Q,\epsilon}^2\|f_\epsilon\|_{Q,\epsilon}^2 \\ 
&\le C^2\sum_{x_1, x_2\in U_\epsilon} \sup_{\|g_\epsilon\|_{Q,\epsilon}=1}\left|\sum_{i=1}^d
\epsilon^j{h_i}(x_1)\overline{h_i}(x_2)Q(x_1,x_2)\right|^2\|g_\epsilon\|_{Q,\epsilon}^2\|f_\epsilon\|_{Q,\epsilon}^2 \\
&\le C^2\sum_{x_1, x_2\in U_\epsilon} \sum_{i=1}^d  
\epsilon^{2j}{h_i}(x_1)^2\overline{h_i}(x_2)^2\left|Q(x_1,x_2)\right|^2\|f_\epsilon\|_{Q,\epsilon}^2.
\end{align*}
In arriving at the last line we have made use of the fact that the $h_i$ are supported on disjoint sets.
Since by Lemma \ref{lem:discretization} we have $\|f_\epsilon\|_{Q,\epsilon}^2\to \|f\|_Q$ as $\epsilon\to 0$, whence
\begin{equation*}
 \lim\sup_{\epsilon\to 0} S_\epsilon^2 \le \sum_{i=1}^dC^2\int\int  |Q(x_1, x_2)|^2 h_i(x_1)^2\overline{h_i}(x_2)^2\ dx_1 dx_2 \|f\|_Q^2 \le C^2q^2d.
\end{equation*}
This gives the needed result.

\end{proof}
\begin{proof}[Proof of Theorem \ref{thm:operator}]
Expanding the definition of $N_d$  we have,
\begin{equation*}
 \begin{split}
  N_d &=\sum_{i=1}^d \left(a^+(g_i)a^+(g_i)+a(g_i)a(g_i)+a^+(g_i)a(g_i)+a(g_i)a^+(g_i)\right)\\
&+\sum_{i=1}^d \left(a_r^+(g_i)a_r^+(g_i)+a_r(g_i)a_r(g_i)+a_r^+(g_i)a_r(g_i)+a_r(g_i)a_r^+(g_i)\right)\\
&-\sum_{i=1}^d \left(2a^+(g_i)a_r^+(g_i)+2a(g_i)a_r(g_i)+a^+(g_i)a_r(g_i)+a(g_i)a_r^+(g_i)\right)\\
&-\sum_{i=1}^d \left(a_r^+(g_i)a(g_i)+a_r(g_i)a^+(g_i)\right).
 \end{split}
\end{equation*}
Here we have used the fact that $a^+(g_i)a_r^+(g_i)= a_r^+(g_i)a^+(g_i)$ and likewise for the left and
right annihilation operators.

For each $i$, we denote by $D_i$ the map on $\mathcal F_Q(\mathcal H)$ given by linear extension of
$f\mapsto g_i\otimes f\otimes g_i$ for $f\in \mathcal H^{\otimes n}$. By the definition of the left
and right annihilation operators,
\begin{align*}
a(g_i)a_r^+(g_i)f &= \left(a(g_i)f\right)\otimes g_i+\mathcal L_{n+2}(T_{2}^{(n+2)}\cdots T_{n+1}^{(n+2)})D_i(f)\\
\text{and}\\
a_r(g_i)a^+(g_i)f &= g_i\otimes \left(a_r(g_i)f\right)+\mathcal R_{n+2}(T_{n}^{(n+2)}\cdots T_{1}^{(n+2)})D_i(f),
\end{align*}
for $f\in \mathcal H^{\otimes n}$, where $\mathcal R_{n+2}$ and $\mathcal L_{n+2}$ are as in Proposition
\ref{prop:contractions}. Now defining
\begin{equation*}
 B_1:=-2d+\sum_{i=1}^d (a(g_i)a^+(g_i)+a_r(g_i)a_r^+(g_i)),
\end{equation*}
we have by Proposition \ref{prop:ancr} that $\|B_1\|\le 2Cq\sqrt{d}$ on $\mathcal F_Q(\mathcal H)\ominus \mathbb C\Omega$. Define also
\begin{equation*}
 B_2:=\sum_{i=1}^d \left(\mathcal R_{n+2}(T_n^{(n+2)}\cdots T_{1}^{(n+2)})D_j(f)+ \mathcal L_{n+2}(T_2^{(n+2)}\cdots T_{n+1}^{(n+2)})D_j(f)\right).
\end{equation*}
By Proposition \ref{prop:contractions}, we have $\|B_2\|\le 2qd$. Finally letting
\begin{equation*}
 B_3:=N_d-2d-B_1+B_2,
\end{equation*}
we have by Lemma \ref{lem:bounds} that $\|B_3\| \le 14 C\sqrt{d}$. This yields an inequality of operators,
\begin{equation*}
 N_d|_{\mathcal F_Q(\mathcal H)\ominus \mathbb C\Omega} \ge 2d(1-q)-2C\sqrt{d}q-14C\sqrt{d}.
\end{equation*}
The expression on the right is positive for sufficiently large $d$.
\end{proof}

\begin{remark}
We have assumed throughout that $q:=\sup_{x,y\in U}|Q(x,y)|<1$. However, we can easily extend the construction
to the case of $q=1$. Write $U=\bigcup_{i\in I} B(x_i, r_i)$ where $B(x_i, r_i)$ denotes the open ball of 
radius $r_i$ centered at $x_i\in \mathbb R^j$. For each $N$, define $U_N:=\bigcup_{i\in I} B\left(x_i, \frac{N-1}{N}r_i\right)$.
Then 
\begin{equation*}
\sup_{x,y\in U_N}|Q(x,y)|\le \sup_{x,y\in \overline{U_N}}|Q(x,y)| < 1,
\end{equation*}
so we can define $\mathcal H_N:=L^2(U_N)$ and apply the construction to get a factor 
$\Gamma_Q(\mathcal H_N)$. Moreover, we have a natural inclusion $\Gamma_Q(\mathcal H_N)\subseteq \Gamma_Q(\mathcal H_{N+1})$, so we can define $\bigcup_{N\in\mathbb N} \Gamma_Q(\mathcal H_N)$. The Fock space $\mathcal F_Q(\mathcal H)$ 
can be constructed by the GNS construction. Finally, by choosing
the functions $g_1,g_2,\ldots$ in Theorem \ref{thm:operator} to be supported in some $U_N$, we see that
we can construct an operator as in Theorem \ref{thm:operator}, so that Corollary \ref{cor:factor} holds as well. 
\end{remark}

\bibliographystyle{plain}
\bibliography{bibliography}

\end{document}